\documentclass{amsart}
\usepackage{amssymb, latexsym}
\usepackage{url}

\usepackage{amsmath}
\usepackage{amsfonts}
\usepackage{amsthm}
\usepackage{textcomp}
\usepackage{graphicx}
\usepackage{mathrsfs}
\usepackage{color}
\usepackage{units}
\usepackage{enumerate}
\usepackage{esint}
\usepackage{caption}
\usepackage{subcaption}

\usepackage[bookmarks, pdftitle={Extremal}, pdfauthor={Michael Kelly}, colorlinks=true, linkcolor=black, citecolor=black, urlcolor=black]{hyperref}

\begin{document}
 \bibliographystyle{plain}
 \newtheorem{theorem}{Theorem}
 \newtheorem{lemma}[theorem]{Lemma}
  \newtheorem{proposition}[theorem]{Proposition}
 \newtheorem{corollary}[theorem]{Corollary}
 \newtheorem{conjecture}[theorem]{Conjecture}
 \newtheorem{definition}{Definition}
  \newtheorem*{problem}{Problem}
   \newtheorem{construction}{Construction}
  \newtheorem{remark}{Remark}
 \newcommand{\mc}{\mathcal}
 \newcommand{\C}{\mathbb{C}}
 \newcommand{\R}{\mathbb{R}}   
 \newcommand{\N}{\mathbb{N}}
 \newcommand{\Q}{\mathbb{Q}}
 \newcommand{\Z}{\mathbb{Z}}
 
       \newcommand{\Rn}{\mathbb{R}^N}
      \newcommand{\Cn}{\mathbb{C}^N}
      \newcommand{\Sn}{\mathrm{S}^{N-1}}
       \newcommand{\Zn}{\mathbb{Z}^N}
       
      \newcommand{\lb}{\left\lbrace}
      \newcommand{\rb}{\right\rbrace}
      \newcommand{\al}{\left\langle}
      \newcommand{\ar}{\right\rangle}
      \newcommand{\ol}{\overline}
      \newcommand{\Real}{\mathrm{Re}}
      \newcommand{\Imaginary}{\mathrm{Im}}
      \newcommand{\sgn}{\mathrm{sgn}}
      \newcommand{\ra}{\rightarrow}
      \newcommand{\dsum}{\displaystyle\sum}
      \newcommand{\dint}{\displaystyle\int}
      \newcommand{\dprod}{\displaystyle\prod}
      
      				\newcommand{\mU}{\mathscr{U}}
				 \newcommand{\lp}{\left(}
				 \newcommand{\rp}{\right)}
				 \newcommand{\RKHS}{reproducing kernel Hilbert space}
 				 \newcommand{\bH}{{\bf H}}
				 \newcommand{\bh}{{\bf h}}  \newcommand{\bu}{\boldsymbol{u}}   							\newcommand{\bv}{\boldsymbol{v}}
				 \newcommand{\balpha}{\boldsymbol{\alpha}}   
				\newcommand{\ds}{\displaystyle}
				\newcommand{\bs}{\boldsymbol}
				\newcommand{\hf }{\nicefrac{1}{2}}
				\newcommand{\hfi }{\nicefrac{i}{2}}
				\newcommand{\length}{\mathrm{Length}}

 \def\today{\number\time, \ifcase\month\or
  January\or February\or March\or April\or May\or June\or
  July\or August\or September\or October\or November\or December\fi
  \space\number\day, \number\year}

\title[Vanishing Problem]{A variation on Selberg's approximation problem}
\author[Kelly ]{Michael~ Kelly}
\date{\today}

\address{Department of Mathematics, Univerisity of Texas, Austin, Texas 78712 USA}
\email{mkelly@math.utexas.edu}
\numberwithin{equation}{section}
\maketitle
\begin{abstract}
 Let $\alpha\in\C$ in the upper half-plane and let $I$ be an interval. We construct an analogue of Selberg's majorant of the characteristic function of $I$ that vanishes at the point $\alpha$. The construction is based on the solution to an extremal problem with positivity and interpolation constraints. 
 \end{abstract}

\section{Introduction}


In the 1970's, Selberg \cite{Sel,V} introduced a useful tool for proving inequalities at the interface of Fourier analysis and number theory. Given a real number $\delta>0$ and an interval $I\subset \R$, he constructed an integrable function $C:\R\ra\R$ that satisfies
	\begin{enumerate}
		\item $C(x)\geq \chi_{I}(x)$,
		\item $\hat{C}(\xi)=0$ whenever $|\xi|>\delta$, and
		\item $\dint_{-\infty}^{\infty}C(x)dx=\mathrm{Length}(I)+\delta^{-1},$
	\end{enumerate}
where $\hat{C}(\xi)$ is the Fourier transform of $C(x)$ (see Section \ref{backgroundSection}). It is not difficult to show, and we will see this below, that (2) implies $C(x)$ is the restriction to $\R$ of an entire of {\it exponential type}. The last property (3) demonstrates that $C(x)$ is a good approximation to $\chi_{I}(x)$ in $L^1 -$norm. In fact, among all integrable functions satisfying conditions (1) and (2) above, the integral appearing in (3) is minimal if, and only if, $\mathrm{Length}(I)\delta\in\Z$. In unpublished work, B. F. Logan \cite{Logan77} found the extremal majorant in the case when $\mathrm{Length}(I)\delta\not\in\Z$ and has shown that the corresponding extremal majorant is unique.  This result has been realized again in the work \cite{Littmann2013} of Littmann, presumably using different methods. \newline
\indent In this paper we study the following variation of this problem suggested to us by E. Bombieri \cite{Bom, Vaa}:
		\begin{quotation} 
		{\it
			Let $\alpha\in\C$ be a point in the upper half-plane. Construct an analogue of Selberg's majorant that vanishes at $\alpha$. 
		}
		\end{quotation}
The motivation for this problem comes primarily from the study of $L$-functions. To illustrate, suppose $L(s)$ is an $L$-function that fails the Riemann hypothesis. This means the function $z\mapsto L(1/2 +iz)$ has a zero at $z=\alpha$, where $0<\Imaginary(\alpha)<1/2$. Suppose further that we wish to use Selberg's majorant in the explicit formula for $L(s)$. We would rather not have the terms in the explicit formula that involve $\alpha$, so it would be desirable to have an analogue of Selberg's function that vanishes at $\alpha$. Applying such a function in the explicit formulas would cause the terms involving $\alpha$ to vanish, but the information is not lost as $\alpha$ is encoded in the function itself.\newline
\indent We do not determine the extremal majorant with the extra vanishing condition in this paper. Rather, we solve a different extremal problem that allows us to produce a class of majorants that are {\it good} approximations {\it and} satisfy vanishing conditions.  Our method allows us to obtain non-trivial bounds on the quantity
	 \[
	 	\rho(\alpha,I,\delta)=\inf \dint_{-\infty}^{\infty} \lb G(t) - \chi_{I}(t)  \rb dt
	\] 
where the infimum is taken over entire functions $G(z)$ that satisfy: 
	\begin{enumerate}[(i)]
		\item $G(t)\geq \chi_{I}(t)$ for all $t\in\R$,
		\item $\hat{G}(\xi)=0$ whenever $|\xi|>\delta$, and
		\item $G(\alpha)=0$.
	\end{enumerate}
Our first result concerns non-trivial bounds on $\rho(\alpha, I,\delta)$.	
\begin{theorem} \label{rhoTheorem}
	Let $\delta>0$, $\alpha\in\C$ with $\Imaginary(\alpha)>0$, and $I\subset \R$ be an interval. Then
		\begin{equation}
			\delta^{-2} \ll \rho(\alpha,I,\delta) \ll \delta^{-3}
		\end{equation}
	as $\delta\ra 0$, and
		\begin{equation}
			\rho(\alpha,I,\delta) \approx \delta^{-1}
		\end{equation}
	as $\delta\ra \infty$. The implied constants are effective and depend only on $\alpha$ and $I$.
\end{theorem}
The upper bounds in this theorem are obtained by constructing an entire function $G(z)$ satisfying the conditions (i)-(iii) above. This construction is based on the solution to the following extremal problem.
\begin{problem} \label{interpolate}
	Let $\delta>0,$ $\alpha,\beta\in\C$ with $\Imaginary(\alpha)>0$. Determine the value of
		\begin{equation}\label{kappa}
			\kappa(\alpha,\beta,\delta)=\inf \dint_{-\infty}^{\infty} F(x)dx
		\end{equation}
	where the infimum is taken over continuous functions $F:\C\ra \C$ with the following properties:
		\begin{enumerate}[(i)]
			\item $F(x)$ is real valued and integrable on  $\R$,
			\item $F(x)\geq 0$ for each $x$ in $\R$,
			\item $F(\alpha)=\beta$
			\item $\hat{F}(\xi)$ is supported in the interval $[-\delta,\delta]$.
		\end{enumerate}
	In addition to finding the minimal integral, find explicit extremal functions for which the minimal integral is obtained.
\end{problem}
We say a function $F(z)$ is admissible for $\kappa$ if it satisfies the conditions in the above problem. Our main result is the solution to this problem.
\begin{theorem}\label{interpolateTheorem}
	Let $\delta>0$, $\alpha\in\C$ be a point in the upper half-plane, and $\beta\in\C$. Then
		\begin{enumerate}[(i)]
			\item $\kappa(\alpha,\beta,\delta)=|\beta|\kappa(\alpha,\beta/|\beta|,\delta)$,
			\item $ \kappa(\alpha,\beta,\delta)=r^{-1}\kappa(r\alpha,\beta,r\delta)$ for each $r>0$, 
			\item $\kappa(\alpha,\beta,\delta)=\kappa(\alpha+t,\beta,\delta)$ for each $t\in\R$, and
			\item 
				\[
					\dfrac{\kappa(\alpha,\beta,\delta)}{2}=\dfrac{ \;\;\; |\beta| K(\alpha,\alpha)  - \delta\Real(\beta) \;\;\; }{K(\alpha,\alpha)^{2} - \delta^{2}  }
				\]
		\end{enumerate}
		where $K:\C\times\C\ra\C$ is given by \[K(\omega,z)=\dfrac{\sin\pi\delta(z-\ol{\omega})}{\pi(z-\ol{\omega})}.\] 
	The infimum in (\ref{kappa}) is achieved for a unique admissible function $F(z)=U(z)\ol{U(\ol{z})}$ where 
		\[
			U(z)=\lambda_{1}K(\alpha,z)+\lambda_{2}K(\ol{\alpha},z),
		\]
	 and $\lambda_{1}$ and $\lambda_{2}$ are given by
		\begin{equation}\label{constants}
			\lambda_{1}=\dfrac{\beta K(\alpha,\alpha)-\delta}{K(\alpha,\alpha)^{2} - \delta^{2}}\;\;\text{ and }\;\;\lambda_{2}=\dfrac{ K(\alpha,\alpha)-\beta\delta}{K(\alpha,\alpha)^{2} - \delta^{2}}.
		\end{equation}
\end{theorem}

\begin{remark}
If $\alpha=x+iy$ and $\beta=b$ is real, then the formula in (iv) becomes
	\[
		\kappa(x+iy,b,\delta)=\dfrac{4\pi y |b|}{\sinh(2\pi y\delta )+\sgn(b)2\pi y\delta}.
	\]
We will be primarily concerned with the case when $b=-1$. 
\end{remark}
Let $F(z;\alpha,\beta)$ be the unique extremal function from Theorem {\ref{interpolateTheorem}}, notice that by using this notation we suppress the dependence of $F$ on $\delta$. Consider the following modification of Selberg's function
	\begin{equation}
		z\mapsto G_{\alpha}(z)= C(z)+F(z;\alpha,-C(\alpha)) .
	\end{equation}
Seeing that $F(x;\alpha,\beta)\geq 0$ for real $x$, it follows that  $C(x)\leq G_{\alpha}(x)$ for each real $x$. Additionally, $G_{\alpha}(\alpha)=0$, and $\hat{G}_{\alpha}(\xi)=0$ whenever $|\xi|>\delta$. So $G_{\alpha}(z)$ is an admissible function for $\rho(\alpha,I,\delta)$. \newline
\indent There is a philosophical problem with $G_{\alpha}(z)$, namely that you require knowledge of the value $C(\alpha)$ to define it, and this is a value we may be trying to avoid. Here is a way to get around this issue. Define a new function $C_{\alpha}(z)$ by
	\begin{equation}
		z\mapsto C_{\alpha}(z)=C(z)\lp 1+F(z;\alpha,-1)  \rp.
	\end{equation}
This function is admissible for $\rho(\alpha,I,2\delta)$ and has the virtue that it doesn't depend on $C(\alpha)$. Moreover, we can produce a majorant that vanishes at many distinct points by repeating this procedure many times. Indeed, for any $\alpha_{1},...,\alpha_{N}\in\C$ in the upper half-plane the function
	\begin{equation}\label{manyPts}
		z\mapsto C(z)\dprod_{n=1}^{N}\lp 1+F(z;\alpha_{n},-1)  \rp
	\end{equation}
is admissible for $\rho(\alpha_{n},I,N\delta)$ for $n=1,...,N$. \newline

The reader may notice that the function $C(z)$ plays a rather passive role in this game. If $M(z)$ is an entire function such that $M(x)\geq m(x)$ for each real $x$, then 
	\[
		z\mapsto M(z)\lp 1+F(z;\alpha,-1)  \rp
	\]
vanishes at $\alpha$ while remaining a majorant of $m(x)$. The problem of determining the extremal majorants $M(z)$ of a function  $m(x)$ is well-studied (see \cite{BMV,CL,CL2,CLV,CV2010,CVIII,GV,HV,LV,M1978,V}). The above modification for prescribed vanishing applies to the class of functions studied in this collection of papers. So the the virtue of this construction is in its adaptability. We know of two cases where extremal majorants with a vanishing constraint have been determined: (1) Vaaler \cite{Vaa} has determined the best majorant of the Kronecker delta function on $\R$ that vanishes at a point $\alpha$ in the upper half-plane, and (2) Littmann-Spanier \cite{LS} have determined the best majorant of the signum function that vanishes at a prescribed point on the imaginary axis. It would be interesting to determine extremal majorants with several vanishing constraints, even for the Kronecker delta function. Such a majorant would have applications in the theory of $L-$functions and their related objects.
\subsection{Organization of paper}
Before we jump into the proofs of our main results, we review some basic information about entire functions of exponential type. Then in section \ref{interpolateSection} we prove Theorem \ref{interpolateTheorem} and compute the Fourier transform of the extremal function $F(z;\alpha,\beta)$. Section \ref{vanishingImaginarySection} contains an analysis of the extremal function from Theorem \ref{interpolateTheorem} in the special case when $\alpha$ is purely imaginary and $\beta=-1$. In section \ref{rhoSection}, we determine a zero free region for Selberg's majorant and prove Theorem \ref{rhoTheorem}. Finally, in the last section we discuss minorants, and generalizations of Theorem \ref{interpolateTheorem} in de Branges spaces. Here we also discuss Theorem \ref{linLemma}, a result in de Branges space that may be of independent interest. Namely, if $E(z)$ is a de Branges function $E(z)$ of bounded type, then we give necessary and sufficient conditions for a de Branges reproducing kernel $K_{E}(\alpha,z)$ to be linearly independent with $K_{E}(\ol{\alpha},z)$.
\subsection{Acknowledgements}
The author would like to thank Emanuel Carneiro, Bill Beckner, Kannan Soundararajan, for their comments and suggestions. He would especially like to thank Enrico Bombieri for suggesting the problem and for his hospitality at the IAS. Most of all, he would like to thank Jeff Vaaler for his guidance and constant encouragement.
%
%
\section{Background}\label{backgroundSection}

Throughout this paper $x$ will be an element of the real numbers $\R$, $z$ will be an element of the complex numbers $\C$, and $\ol{z}$ is the complex conjugate of $z$. $\mU=\lb z\in\C:\Imaginary(z)>0 \rb$ will denote the upper half-plane of $\C$, where $\Imaginary(z)$ is the imaginary part of $z$.  If $F(z)$ is an entire function, the complex conjugate $ F^{*}(z)$ of $ F(z)$ defined by $ F^{*}(z)=\ol{F(\ol{z})}$ is also an entire function. If $I\subset \R$ is an interval, then $\chi_{I}(t)$ will denote the characteristic function of $I$. If $F(z)=F^{*}(z)$, then we say that $F(z)$ is {\it real entire}. An integrable function $ F(x)$ will be called {\it admissible} if $F(x)$ satisfies the conditions of the extremal problem under consideration (not necessarily extremality). $ F(x)$ will be called {\it extremal} if it is admissible and achieves the extreme value defined by the extremal problem.  If $f(t)$ is square integrable and continuous function on $\R$, then the Fourier transform $\hat{f}(\xi)$ is defined by
	\[
		\hat{f}(\xi)=\lim_{T\ra\infty}\dint_{-T}^{T}e^{-2\pi it\xi }f(t)dt.
	\]
We extend the definition of the Fourier transform to all suitably nice functions in the usual way.\newline 
\indent Suppose $F:\R\ra\C $ is a integrable and $\hat{F}(\xi)=0$ whenever $|\xi|>\delta$
	\begin{equation}\label{extension}
		z\mapsto \dint_{-\delta}^{\delta}e(z\cdot\xi)\hat{ F}(\xi) d\xi
	\end{equation}
is equal to $ F(x)$ for almost every $ x$ in $\Rn$ and for each closed curve $\gamma$ in $\C$ we have 
	\[
		\dint_{\gamma} \dint_{-\delta}^{\delta}e(z\xi)\hat{ F}(\xi) d\xi dz = \dint_{-\delta}^{\delta} \dint_{\gamma}e(z\xi)dz\hat{F}(\xi) d\xi=0
	\]
where we have used Fubini's theorem to interchange the order of integration. Since the curve $\gamma$ was arbitrary, (\ref{extension}) defines an entire function by Morera's theorem. Consequentially, $ F(x)$ is almost everywhere equal to the restriction to $\R$ of an entire function. We will always identify admissible functions with their extensions to entire functions. The representation (\ref{extension}) shows that admissible functions satisfy the following growth estimate in $\C$
	\begin{equation}\label{growthEstimate}
		|F(z)|\ll_{\epsilon} e^{2\pi (1+\epsilon) \delta\Imaginary(z)}
	\end{equation}
for each $\epsilon>0$. Entire functions that satisfy an estimate of the form (\ref{growthEstimate}) are called {\it entire functions of exponential type.}  An entire function $F(z)$ is said to be of exponential type $2\pi\sigma>0$ if 
	\[
		|F(z)|\ll_{\epsilon} e^{2\pi\sigma(1+\epsilon) |z|}
	\]
for each $\epsilon>0$. From the definition it is immediate that an entire function of exponential type is an entire function of order 1, however, the converse does not hold. The Riemann Xi function is an example of an entire function of order 1 that is not of exponential type. \newline
\indent In this paper we are primarily concerned with entire functions of exponential type which are bounded on the real axis. Although this subject is well studied,\footnote{We refer the interested reader to \cite{B} for further information.} we will collect some relevant material here. Throughout the paper we will use the notation of Stein \cite{Stein57,SW}. $E_{\sigma}$ will denote the space of entire functions with exponential type at most $\sigma$, and $B_{\sigma}$ will be the subspace of $E_{\sigma}$ consisting of functions which are bounded on the real axis. For $1\leq p\leq\infty$ we let $L^p$ be the space of entire functions $F(z)$ such that 
	\[
		\|F\|_{p}=\begin{cases} \;\;\; \lb  \dint_{-\infty}^{\infty}|F(x)|^{p} \rb^{1/p} & \text{ if } 1\leq p<\infty \\
				                        \;\;\;  \ds\sup_{-\infty<x<\infty}|F(x)| & \text{ if } p=\infty
				\end{cases}
			\]
is finite. The following classical theorem of Paley and Wiener gives two equivalent ways of looking at $L^{2}\cap E_{\pi\sigma}$.
\begin{theorem}[Paley-Wiener]
If $F\in L^{2}\cap E_{\pi\sigma}$, then $\hat{F}(\xi)=0$ if $|\xi|>\sigma/2$. Conversely, any function $F(x)$ that is square-integrable on $\R$ and that satisfies $\hat{F}(\xi)=0$ for $|\xi|>\sigma/2$ is a.e. equal to the restriction to the real axis of a function in $F\in L^{2}\cap E_{\pi\sigma}$.
\end{theorem}

The space of functions characterized in the Paley-Wiener theorem forms a Hilbert space, with respect to the $L^{2}(\R)-$inner-product $\al\cdot,\cdot\ar$, called the {\it Paley-Wiener space of type $\sigma/2$.} We will use the notation 
\[  {\bf H}_{\sigma}=L^{2}\cap E_{2\pi\sigma}\cong \lb f\in L^{2}(\R)~:~\hat{f}(\xi)=0\text{ if }|\xi|>\sigma\rb \]	
when referring to this space. Notice that ${\bf H}_{r}$ is the image of $L^{2}([-r,r])$ through the Fourier transform. This implies that the image of the standard basis through the Fourier transform
	\[
		\lb e(n\xi/\sigma) \; :\; n\in\Z   \rb \ra \lb  \dfrac{\sin2\pi \sigma (z-n/\sigma)}{\pi \sigma(z-n/\sigma)} \; : \; n\in\Z \rb
	\]
 is an orthogonal basis for ${\bf H}_{\sigma}$. From this simple observation we obtain the following interpolation result: {\it for each $F\in H_{\sigma}$
 	\[
		F(z)=\dsum_{n\in\Z} F(n/\sigma)  \dfrac{\sin2\pi \sigma (z-n/\sigma)}{\pi \sigma(z-n/\sigma)}
	\]
where the sum converges in ${\bf H}_{\sigma}$.} It follows that an element of ${\bf H}_{\sigma}$ is completely determined by the values it takes on the lattice $\sigma^{-1}\Z$ or its translates. Another important property of ${\bf H}_{\sigma}$, that we will use extensively, is the reproducing kernel property: for every $F\in {\bf H}_{\sigma}$ and $\omega\in\C$
	\begin{equation}\label{evaluation}
		|F(\omega)| \leq 2\sigma \| F\|_{2}.
	\end{equation}
This shows that evaluation is a continuous linear functional. It then follows from the Riesz representation theorem that evaluation can be realized as the inner product with an element of ${\bf H}_{\sigma}$. This element is called the reproducing kernel, and for ${\bf H}_{\sigma}$ it is given by
	\begin{equation}
		K(\omega,z)=\dfrac{\sin2\pi\sigma(z-\ol{\omega})}{\pi(z-\ol{\omega})}.
	\end{equation}
So (\ref{evaluation}) implies that for each $F\in {\bf H}_{\sigma}$ and $\omega\in\C$
	\[
		F(\omega)=\al  F, K(\omega,\cdot) \ar.
	\]

We close this section with several inequalities that we will use in the sequel. 
\begin{lemma}[Bernstein] \label{BernsteinLemma}
	If $F\in B_{\sigma}$, then $\| F^{\prime} \|_{\infty}\leq \sigma \| F\|_{\infty}$.
\end{lemma}
Let $L^{p}$ be the space of entire functions $F(z)$ for which $\|F\|_{p}$ is finite. It is known (see \cite{B}) that if $F\in E_{\sigma}\cap L^p$ for $p\geq1$, then $\| F^{\prime} \|_{p}\leq \sigma \| F\|_{p}.$ The following can be found in \cite[p.~83]{B}.
	\begin{proposition} \label{BoasLemma}
		 If $f\in B_{\sigma}$, then
		  \begin{equation} \label{boasIneq}
			|f(x+iy)|\leq \|f\|_{\infty} \cosh \sigma y
		  \end{equation}
	    for all real numbers $x$ and $y$.
	\end{proposition}
\begin{lemma}[Stein, \cite{Stein57}]\label{SteinLemma}
	If $f\in E_{\sigma}\cap L^{p}$, $1\leq p<\infty$, then for a universal $c>0$
		\[ \| f\|_{\infty} \leq c\sigma^{1/p}\| f\|_{p}.\]
\end{lemma}	    

%
%

\section{Interpolation in the upper half-plane \\ Proof of Theorem \ref{interpolateTheorem}}\label{interpolateSection}
In this section we prove Theorem \ref{interpolateTheorem}. The proof has four steps:
	\begin{enumerate}
		\item reduce the $L^1$ variational problem to an $L^{2} $ variational problem by showing that each admissible function $F(z)$ can be factored as $F(z)=U(z)\ol{U(\ol{z})}=U(z)U^{*}(z)$.
		\item realize that $U(z)$ is an element of the Paley-Wiener space ${\bf H}_{\delta/2}$, and that $U(z)=\al U, K(z,\cdot) \ar$ for any $z\in\C$.
				\item Reformulate the problem as minimizing $\| U \|^{2}_{2}$ subject to the condition  
				\[
					\al U,K(\alpha,\cdot)  \ar \ol{\al U,K(\ol{\alpha},\cdot)  \ar}=\beta.
				\]
			This is just another way of writing $F(\alpha)=\beta$.
		\item Solve the problem in the previous step in a generic inner product space. 
	\end{enumerate}
Steps (1) and (2) are performed in Proposition \ref{fejerThm} and step (4) is performed in Lemma \ref{extremalHilbert}.
	\begin{proposition}\label{fejerThm}
		Suppose $F(z)\in {\bf H}_{\delta}$ is real valued and non-negative on the real axis and that $F(z)$ is not identically zero. 
		Then there exists an entire function $U(z)\in {\bf H}_{\delta/2}$ such that $U(z)$ is zero-free in $\mU$ and $F(z)=U(z)U^{*}(z)$.
	\end{proposition}

	\begin{proof}
	Let $\lb \omega_{n} : n=1,2...\rb$ be the zeros of $F(z)$, listed with appropriate multiplicity, in the upper half plane and let 
		\[
			B_{N}(z)=\dprod_{n=1}^{N}\dfrac{1-z/\ol{\omega_{n}}}{1-z/\omega_{n}}.
		\]
	We define a sequence of entire functions $F_{N}(z)$ by $F_{N}(z)=B_{N}(z)F(z)$. Each of the functions $F_{N}(z)$ is in ${\bf H}_{\delta}$ by the Paley-Wiener theorem. Notice that $\|F\|=\|F_{N}\|$ for each $N$ because $|B_{N}(x)|=1$ for each $x\in\R$. Thus, by the Banach-Alaoglu theorem, the sequence $\lb F_{N}(z) \rb$  is weakly compact in $H_{\delta}$. Since this space is a reproducing kernel space, it follows that $F_{N}(z)\ra G(z)$ pointwise for a subsequence. Observe that $|B_{N}(z)|\geq 1$ if $z\in\mU$. It follows that $G(z)$ is zero free in $\mU$ and that $|G(t)|=|F(t)|$ for real $t$. This shows that $F(z)^{2}=F(z)F^{*}(z)=G(z)G^{*}(z)$. In particular, the non-real zeros of $G(z)$ occur with even multiplicity. \newline
\indent Since $F(z)$ is real valued and non-negative on $\R$, the zeros of $G(z)$ occur with even multiplicity and so there is an entire function $U(z)$ for which $G(z)=U(z)^{2}$. Then $F(z)^{2}=\lb U(z)U^{*}(z) \rb^{2}$ and since $F(z)$ is real valued and non-negative on $\R$ it follows that $F(z)=U(z)U^{*}(z)$. \newline
\end{proof}

      \begin{lemma} \label{extremalHilbert}
        Let $\bf H$ be a complex vector space with inner product $\al\cdot,\cdot\ar$, $\beta\in\C$. Let $\bu,\bv\in{\bf H}$ be linearly independent, $\eta=\|\bu\|\|\bv\|$, and $\nu=\al \bu,\bv\ar$. If
          \begin{equation} \label{Hcondition}
            \al{\bf h},\bu\ar\overline{\al {\bf h},\bv\ar}=\beta
            \end{equation}
      then
            \begin{equation} \label{equal}
                 \|{\bf h}\|^{2}\geq2\dfrac{|\beta|\eta-\Real(\beta\nu)}{\eta^{2}-|\nu|^{2}}.
            \end{equation}  
        If ${\bf h}$ satisfies (\ref{Hcondition}), then $\bf h$ achieves equality in (\ref{equal}) if and only if
            \begin{equation} \label{hilbertExtremal}
                  \omega{\bf h}= \lambda_{1}\bu +\lambda_{2}\bv
            \end{equation}
        for some $|\omega|=1$, where
            \begin{equation*}
              \lambda_{1}=\dfrac{\gamma\beta \eta -|\nu|}{\eta^{2}-|\nu|^{2}} \;
              \text{ and } \;
                \lambda_{2}=\dfrac{\eta-\beta \nu}{\eta^{2}-|\nu|^{2}}
              \end{equation*}
        and $\gamma=\nu/|\nu|$.
      \end{lemma}
      \begin{proof}  By scaling considerations, it suffices to prove the claim when $|\beta|=\|\bu\|=\|\bv\|=1$ and $\nu=\overline{\nu}$.\newline
          For any $c_{1}$ and $c_{2}$
            \begin{equation}
              \|{\bf h}-(c_{1}\bu+c_{2}\bv)\|^{2}\geq 0.
            \end{equation}
          Expanding this out gives
            \begin{equation}\label{eq:ineq1}
                  2\Real \Big\lbrace \al {\bf h},c_{1}\bu+c_{2}\bv\ar\Big\rbrace
                  - \|c_{1}\bu+c_{2}\bv\|^{2} \leq \|{\bf h}\|^{2}.
            \end{equation}
          Equality occurs in (\ref{eq:ineq1}) if and only if ${\bf h}=c_{1}\bu+c_{2}\bv$. We let $\bf h$ satisfy $\al{\bf h},\bu\ar\overline{\al{\bf h},\bv\ar}=\beta$ and set
            $$c_{1}=\dfrac{\beta-\nu}{1-\nu^{2}}$$
          and
              $$c_{2}=\dfrac{1-\beta\nu}{1-\nu^{2}}.$$
          It can be checked that
            \begin{equation}
             \|c_{1}\bu+c_{2}\bv\|^{2}=\dfrac{2-2\Real\left(\beta\right)\nu}{1-\nu^{2}}=2\Real(c_{2}).
            \end{equation}
          Scaling $\bf h$ by a suitable constant of absolute value 1 we find that
            \begin{equation*}
                  \al{\bf h},\bu\ar=\beta/r \; \text{ and } \; \al{\bf h},\bv\ar=r
            \end{equation*}
          for some $r>0$. Thus
            \begin{equation*}
                2\Real \Big\lbrace \al {\bf h},c_{1}\bu+c_{2}\bv\ar\Big\rbrace
                =2 \lb \dfrac{\Real(c_{1}\overline{\beta})}{r}+\Real(c_{2})r \rb.
            \end{equation*}
          Seeing that  $\Real(c_{1}\overline{\beta})=\Real(c_{2})$, we have
            \begin{equation}
                 \left( \dfrac{1}{r}+r -1\right)2\Real(c_{2}) \leq \|{\bf h}\|^{2}.
            \end{equation}
          This completes the proof.
      \end{proof}

\begin{proof}[Proof of Theorem \ref{interpolateTheorem}]
  If $F(z)$ satisfies
  	\[
		\dint_{-\infty}^{\infty}F(x)dx<\infty,
	\]
then $F\in L^{1}\cap E_{2\pi\delta}$ and, by Lemma \ref{SteinLemma}, $F\in B_{2\pi \delta}$. It follows that $F\in L^{1}\cap L^{\infty}\cap E_{2\pi\delta}$ which in turn implies that $F\in{\bf H}_{\delta}$. It follows from Proposition \ref{fejerThm} that there exists a function $U\in {\bf H}_{\delta/2}$ such that $F(z)=U(z)U^{*}(z)$. Using this factorization we can write
	\[
		\kappa(\alpha,\beta,\delta)=\inf \| U\|_{2}^{2}
	\]
where the infimum is taken over functions $U\in {\bf H}_{\delta/2} $ such that  
				\[
					\al U,K(\alpha,\cdot)  \ar \ol{\al U,K(\ol{\alpha},\cdot)  \ar}=\beta.
				\]
(By the reproducing property of ${\bf H}_{\delta/2}$, this is just another way of writing $F(\alpha)=\beta$.) The functions $K(\alpha,z)$ and $K(\ol{\alpha},z)$ are linearly independent if, and only if, 
\[ K(\alpha,\alpha)^{2}-|K(\alpha,\ol{\alpha})|^{2}=\lb\dfrac{\sinh 2\pi \delta\Imaginary(\alpha)}{2\pi\Imaginary(\alpha)}\rb^{2}-\delta^{2}\neq0.\]
    This holds if, and only if, $\Imaginary(\alpha)\neq 0$. The result now follows from Lemma \ref{extremalHilbert} with $\bu=K(\alpha,z)$ and $\bv=K(\ol{\alpha},z)$.
\end{proof}

In the following lemma we compute the Fourier transform of $F(z;\alpha,\beta)$. 
     \begin{lemma}
        		Let $K(\omega,z)=\sin(\pi\delta(z-\ol{\omega}))/\pi(z-\ol{\omega})$, and $\mathscr{F}$ denote the Fourier transform, and
                		\begin{equation}
                			G_{\alpha,\omega}(t)=\mathscr{F}\big( K(\alpha,\cdot)K(\omega,\cdot)\big)(t),
                		\end{equation}
        		then
                		\begin{equation*}
                  		 G_{\alpha,\omega}(t) =
                  		e^{-\pi i t(\overline{\omega}+\overline{\alpha})} \dfrac{\sin\big\lbrace \pi (\overline{\omega}-\overline{\alpha})(\delta -|t| )_{+}\big\rbrace}{\pi(\overline{\omega}-\overline{\alpha})}                   		
                  		\end{equation*}
        		Let $F(z)=F(z;\alpha,\beta)$ be  the extremal function identified in Theorem \ref{interpolateTheorem}, then
                		\begin{equation}
                		\mathscr{F}\big( F \big)(t)= (|\lambda_{1}|^{2}+|\lambda_{2}|^{2})G_{\alpha,\overline{\alpha}}(t)+(\ol{\lambda_{1}}\lambda_{2}e(-\alpha)+\lambda_{1}\ol{\lambda_{2}}e(-\ol{\alpha}))(\delta-|t|)_{+},
                		\end{equation}
        		where $\lambda_{1}$ and $\lambda_{2}$ are given by (\ref{constants}), and $(x)_{+}=\max\lb 0,x\rb$.
        \end{lemma}
        \begin{proof}
        If $a\leq b$ and $\xi\in\C$, then
                $$ \dint_{a}^{b} e^{2\pi i s\xi}ds= \dfrac{\sin( \pi\xi(b-a))}{\pi\xi}e^{\pi i (b+a)\xi}. $$
        Now, if $\xi\in\C$ and  $\delta>0$, then
                $$ \dint_{-\infty}^{\infty} e^{2\pi i s \xi} \chi_{[-\delta/2,\delta/2]}(s)\chi_{[-\delta/2,\delta/2]}(t-s)ds
                =\dfrac{\sin \pi\xi(\delta -|t| )_{+}}{\pi\xi}e^{\pi i t\xi}. 
                $$
        This is seen by observing that
            \begin{eqnarray*}
              \chi_{[-\delta/2,\delta/2]}(s)\chi_{[-\delta/2,\delta/2]}(t-s)
              &=& \chi_{[-\delta/2,\delta/2]\cap [t-\delta/2,t+\delta/2]}(s).
            \end{eqnarray*}
            Hence
              \begin{equation*}
                  \dint_{-\infty}^{\infty} e^{2\pi i s \xi} \chi_{[-\delta/2,\delta/2]}(s)\chi_{[-\delta/2,\delta/2]}(t-s)ds
                  =\begin{cases}
                      \dint_{t-\delta/2}^{\delta/2} e^{2\pi i s\xi}ds & \text{ if } t\geq 0 \\ \\
                      \dint_{-\delta/2}^{t+\delta/2} e^{2\pi i s\xi}ds & \text{ if } t <0
                    \end{cases}.
              \end{equation*}
            When $t\geq 0 $
                  \begin{eqnarray*}
                   \dint_{t-\delta/2}^{\delta/2} e^{2\pi i s\xi}ds
                    &=& \dfrac{\sin \pi\xi(\delta-t)}{\pi\xi}e^{\pi i t\xi}
                   \end{eqnarray*}
            and when $t<0$
                  \begin{eqnarray*}
                   \dint_{-\delta/2}^{t+\delta/2} e^{2\pi i s\xi}ds
                    &=&  \dfrac{\sin \pi\xi(\delta-|t|)}{\pi\xi}e^{\pi i t\xi}.
                   \end{eqnarray*}
            Finally if $|t|<\delta$, then
              \begin{eqnarray*}
                \dint_{-\infty}^{\infty} e^{-2\pi i s t } K_{\alpha}(s)K_{\omega}(s)ds
                &=& e^{-2\pi i \overline{\omega}t}\dint_{-\infty}^{\infty}e^{2\pi i(\overline{\omega}-\overline{\alpha})s}\chi_{[-\delta/2,\delta/2]}(s)\chi_{[-\delta/2,\delta/2]}(t-s)ds \\
                &=&e^{-2\pi i \overline{\omega}t}e^{\pi i t(\overline{\omega}-\overline{\alpha})} \dfrac{\sin \pi (\overline{\omega}-\overline{\alpha})(\delta -|t| )}{\pi(\overline{\omega}-\overline{\alpha})} \\
                &=& e^{-\pi i t(\overline{\omega}+\overline{\alpha})} \dfrac{\sin \pi (\overline{\omega}-\overline{\alpha})(\delta -|t| )}{\pi(\overline{\omega}-\overline{\alpha})}.
              \end{eqnarray*}
            The rest of the proof is straightforward.
        \end{proof}

\section{Vanishing on the imaginary axis}\label{vanishingImaginarySection}
As before, we let $F(z;\alpha,\beta)$ be the extremal function coming from Theorem \ref{interpolateTheorem}.  In this section we analyze $F(z;\alpha,\beta)$ in the special case $\alpha=ir$ for some $r>0$ and $\beta=-1$. The estimates we obtain extend to $\alpha\in\mU$ by the identity 
	\[
		F(z;\alpha,\beta)=F(z-\Real(\alpha);i\Imaginary(\alpha),\beta),
	\]
which follows from item {\it (iii)} from Theorem \ref{interpolateTheorem}.\newline
\indent In the case when $\alpha=ir$ and $\beta=-1$, the function $U(z)$ given in Theorem \ref{interpolateTheorem}  is given by\footnote{Note that $K(ir,ir)=K(-ir,-ir)=(2\pi r)^{-1}\sinh(2\pi r\delta)$ and $K(ir,-ir)=\delta$}
	\begin{equation}\label{u}
		U(z)= \dfrac{ K(ir,z)-  K(-ir,z)}{K(ir,ir)- K(-ir,ir)},
	\end{equation}
where
	\[
		K(\omega,z)=\ds\int_{-\delta/2}^{\delta/2}e^{-2\pi i(z-\ol{\omega})\xi}d\xi=\dfrac{\sin \pi\delta(z-\ol{\omega})}{\pi(z-\ol{\omega})}.
	\]
We record some basic estimates for $U(z)$ in the following lemma.
\begin{lemma}\label{lemmaU}
	For each $r>0$, the function $U(z)$ defined by (\ref{u}) has the following properties
	\begin{enumerate}[(i)]
		\item $\hat{U}(\xi)=0$ if $|\xi|>\delta/2$,
		\item $U(ir)U^{*}(ir)=- 1$,
		\item $\| U \|_{\infty}^{2}\leq 4\pi r\delta \lp\sinh(2\pi r\delta)- 2\pi r\delta \rp^{-1}$, and
		\item $\|U\|_{2}^{2}= \kappa(ir,-1,\delta)=4\pi r( \sinh(2\pi r\delta)- 2\pi r\delta )^{-1}$
		\item 		\[
						\left|\hat{U}\ast\ol{\hat{U}}(\xi)\right|\leq \|U\|_{2}^{2},
					\]		
	\item	\[
			|U(z)|\leq \dfrac{\|U\|^{2}}{|z-ir|}\exp\lp{\pi\delta\Imaginary(z\pm ir)}\rp
		\]
	where $\pm$ is chosen according to whether $\Imaginary(z)>0$ or not.
	\end{enumerate}
\end{lemma}
\begin{proof}
Item (i) follows from the definition of $K(\omega,z)$ and linearity of the Fourier transform. Item (ii) is easy to verify by using the definition of $U(z)$ and noticing that $U(ir)=- U^{*}(ir)$=1. For the remaining items, we use the following facts: (a) the function $t \mapsto K(\omega,t)$ is square integrable on $\R$ and (b),

	\[
		K(\omega,\lambda)=\al K(\omega,\cdot),K(\lambda,\cdot) \ar=\ol{K\lp \ol{\lambda},\ol{\omega}\rp} 
	\]
where $\al \cdot,\cdot \ar$ is the usual $L^{2}-$inner product on $\R$. To compute $\|U\|_{2}^{2}$ we write $\| U\|_{2}^{2}=|\al U,U \ar|^{2}$ and use (\ref{u}):
	\begin{eqnarray*}
		\| U \|_{2}^{2}
		&=& 2\dfrac{ K(ir,ir)- K(ir,-ir)}{\lp K(ir,ir)-K(ir,-ir) \rp^2} \\
		&=& \dfrac{ 4\pi r}{ \sinh(2\pi r\delta)- 2\pi r\delta }. 
	\end{eqnarray*}
By the reproducing property we see that $U(\omega)=\al U, K(\omega,\cdot) \ar$. It follows from the Cauchy-Schwarz inequality that
	\[
		|U(t)|^{2} \leq \delta \| U\|_{2}^{2}.
	\]
		To finish off the proof of the lemma, notice that by the Plancherel theorem and Young's inequality we have
		\begin{equation}
			\left|\hat{U}\ast\ol{\hat{U}}(\xi)\right|\leq \|U\|_{2}^{2}.
		\end{equation}
	A trivial estimation of $\sin(z)$ yields
		\begin{equation}
			|U(z)|\leq \dfrac{\|U\|_{2}^{2}}{|z-ir|}\exp\lp{\pi\delta\Imaginary(z\pm ir)}\rp
		\end{equation}
	where $\pm$ is chosen according to whether $\Imaginary(z)>0$ or not.
\end{proof}

\section{Selberg's Majorant with Vanishing \\ Proof of Theorem \ref{rhoTheorem}}\label{rhoSection}
In this section we will prove Theorem \ref{rhoTheorem}. We begin by determining a zero free region for Selberg's function $C(z)$. This will serve as a model for how to bound $\rho(I,\alpha,\delta)$ from below. Using this model we will prove the lower bounds appearing in Theorem \ref{rhoTheorem}. This result is a consequence of the fact that admissible functions are Lipschitz on horizontal strips in $\C$ together with a straightforward analysis of the Lipschitz constant. Then we will construct an admissible function for $\rho(I,\alpha,4\delta)$ and estimate its integral.
	\begin{proposition}\label{zeroFree}
		Suppose $\delta\length(I) \ll1$.  If  $C(\omega)= 0$, then $|\omega|= \Omega(\delta^{-1})$. 
	\end{proposition}
We will use the following corollary of Lemma \ref{BernsteinLemma}, Proposition \ref{BoasLemma}, and Lemma \ref{SteinLemma}.

\begin{corollary}\label{LipschitzCorollary}
	Let $F\in E_{\sigma}\cap L^{1}$ such that $F(0)\geq 1$ and $\|F\|_{1}=L$. Then if $\omega$ is a zero of $F(z)$, then
		\[
				1\ll L\sigma^{2}|\omega|\cosh(\sigma\Imaginary(\omega))
		\]
	 where the implied constant is absolute.
\end{corollary}	
\begin{proof}
	Suppose $F(\omega)=0$. By the mean value theorem 
		\begin{equation}\label{MVT}
			|F^{\prime}(u\omega)|=\left| \dfrac{F(\omega)-F(0)}{\omega}  \right|\geq \dfrac{1}{|\omega|}
		\end{equation}	
	for some $0\leq u\leq 1$. Combining Proposition \ref{BoasLemma}, Lemma \ref{BernsteinLemma}, and then Lemma \ref{SteinLemma} we have
		\begin{equation}\label{amal}	
			|F^{\prime}(u\omega)|\leq \cosh(\sigma \Imaginary(\omega))\|F^{\prime}\|_{\infty}\leq c\sigma^{2}L\cosh(\sigma \Imaginary(\omega)).
		\end{equation}
	Combining (\ref{MVT}) with (\ref{amal}), we are lead to the inequality
		\[
			1\leq c|\omega|\sigma^{2}L\cosh(\sigma \Imaginary(\omega))
		\]
	where $c>0$ is the constant appearing in Lemma \ref{SteinLemma}.
\end{proof}

      \begin{proof}[Proof of Proposition \ref{zeroFree}]
      		Suppose $C(\omega)=0$ and that $0\in I$. By Corollary \ref{LipschitzCorollary}
			\begin{equation}\label{inequality2}
				1\ll \lp \length(I)+\delta^{-1}\rp(2\pi\delta)^{2}|\omega|\cosh2\pi \delta\Imaginary(\omega).
			\end{equation}
		Combining the fact that $\delta\length(I) \ll1$ and (\ref{inequality2}) yields
			\[
				1\ll 2\pi\delta|\omega|\cosh2\pi \delta|\omega|,
			\]
		it follows that $\delta|\omega|=\Omega(1)$. 
      \end{proof}
      This demonstrates that vanishing doesn't come for free with extremal majorants of $\chi_{I}(t)$ and that forcing vanishing will come at some cost for small values of $\delta$.
       We will now consider how well Selberg type functions, which have the property that they vanish at a prescribed point in the upper half plane, can approximate $\chi_{I}(t)$. If $\delta$ is small, then we know Selberg's function is not admissible and we can use the modifications suggested in the introduction to obtain the upper bounds on $\rho(I;\alpha,\delta)$. We are now in a position to prove the first part of Theorem \ref{rhoTheorem}.

       \begin{proof}[Proof of Theorem \ref{rhoTheorem} (lower bound):]
		  Following the proof of Proposition \ref{zeroFree} we find that
			    \begin{equation}
				 1\ll \rho(I;\alpha,\delta)(2\pi\delta)^{2}|\alpha|\cosh2\pi \delta\Imaginary(\alpha).
			    \end{equation}
		  Because $\cosh(2\pi\delta\Imaginary(\alpha))\sim 1$ when $\delta\ll 1$ we have 
			    \begin{equation}
				  \delta^{-2} \ll_{\alpha}  \rho(I;\alpha,\delta)
			    \end{equation}
		    when $\delta\ll 1$. \newline
		    \indent Now, when $\delta\gg 1$, we know that $\rho(I;\alpha,\delta)\gg \delta^{-1}$ because Selberg's function (while not admissible) satisfies weaker assumptions than admissible functions and has $L^{1}$ distance $\delta^{-1}$ from the characteristic function of $I$.
      \end{proof}

As before, we let $F(z;\alpha,\beta)$ denote the extremal function described in Theorem \ref{interpolateTheorem}. To produce the upper bound in the theorem, we will use the function
	\[
		C_{\alpha}(z)=C(z)\lp1+F(z;\alpha,-1)\rp,
	\]
which is admissible for $\rho(I;\alpha,\delta)$. 

\begin{proof}[Proof of Theorem \ref{rhoTheorem} (upper bound):]
We begin with the observation that 
	\begin{equation}\label{rhoUpper}
		\rho(I;\alpha,2\delta)\leq \dint_{-\infty}^{\infty} C_{\alpha}(t)dt-\length(I).
	\end{equation}
	Now when $\delta\ll 1$, we estimate the right hand side of (\ref{rhoUpper}) by
		\[
			\dint_{-\infty}^{\infty} C_{\alpha}(t)dt\leq \|C\|_{1}(1+\|F\|_{\infty}).
		\]
	 This inequality, combined with (\ref{rhoUpper}), and Lemma \ref{lemmaU}  yields
		\[
			\rho(I;\alpha,2\delta)\leq  (\length(I)+\delta^{-1})\lp1+ \dfrac{8\pi^{2}c\delta \Imaginary(\alpha)}{ \sinh(2\pi \Imaginary(\alpha)\delta)- 2\pi \Imaginary(\alpha)\delta } \rp
		\]
	and as $\delta\ra0$ this reduces to 
		\[
			\rho(I;\alpha,2\delta)\leq O(1)+3(2\pi c +o(1)) \dfrac{4\pi \Imaginary(\alpha)}{(2\pi\Imaginary(\alpha)\delta)^{3}}\ll_{\alpha}\delta^{-3}.
		\]
	In the case when $\delta\gg 1$, we use the uniform estimate
		\[
			F(t;\alpha,-1)\leq \dfrac{4\pi \Imaginary(\alpha)\delta}{ \sinh(2\pi \Imaginary(\alpha)\delta)- 2\pi \Imaginary(\alpha)\delta } \ll_{\alpha} e^{-2\pi \Imaginary(\alpha)\delta}
		\]
	from Lemma \ref{lemmaU}. Combining this with (\ref{rhoUpper}) we have
		\begin{eqnarray*}
			\rho(I;\alpha,2\delta) 
			&\ll_{\alpha}& \lp 1+ e^{-2\pi \Imaginary(\alpha)\delta}\rp\lp\length(I)+\delta^{-1}\rp -\length(I) \\
			&\ll_{\alpha}& \length(I)e^{-2\pi \Imaginary(\alpha)\delta} +\delta^{-1}e^{-2\pi \Imaginary(\alpha)\delta}+\delta^{-1} \\
			&\ll_{\alpha,I}& \delta^{-1}.
		\end{eqnarray*}
\end{proof}

\section{Remarks and generalizations}\label{RandG}

\subsection*{A word on minorants}
In the literature, when speaking of Beurling-Selberg majorants, it is customary to also speak of Beurling-Selberg {\it minorants.} We have decided to say just a few words in this direction, owing to the fact that we have not even determined the extremal majorant of the characteristic function of the interval with a vanishing constraint. That being said, we can construct minorants in the same way that we constructed majorants by modifying Selberg's functions.\newline
\indent Besides the majorant $C(z)$, Selberg constructed an analogous minorant $c(z)$. That is, for each real number $\delta>0$ and an interval $I\subset \R$, he constructed an integrable function $c:\R\ra\R$ that satisfies
	\begin{enumerate}
		\item $c(x)\leq \chi_{I}(x)$,
		\item $\hat{c}(\xi)=0$ whenever $|\xi|>\delta$, and
		\item $\dint_{-\infty}^{\infty}c(x)dx=\mathrm{Length}(I)-\delta^{-1}.$
	\end{enumerate}
The minorant $c(x)$ extends to an entire function of exponential type $2\pi\delta$, like the majorant $C(z)$.  A notable difference between $C(z)$ and $c(z)$ is that $c(z)$ becomes a worse $L^{1}$-approximation than the constant function $f(x)=0$ when $\delta < \length(I)^{-1}$. So the reader may think of $\delta$ as being  large when we are discussing Selberg's minorant.  \newline 
\indent Observe that the function 
	\[
		z \mapsto c(z)-F(z;\alpha,c(\alpha))
	\]
is an entire function of exponential type $2\pi\delta$ that minorizes the characteristic function of $I$ on the real line, and  vanishes at $\alpha$. Similarly, if $F(z;\alpha,1)\leq 1$, then 
	\[
		z\mapsto c(z)\lp 1- F(z;\alpha,1)\rp
	\]
is an entire function of exponential type $4\pi\delta$ that minorizes the characteristic function of $I$ on the real line, and that vanishes at $\alpha$. By Lemma \ref{lemmaU}, the condition that $F(z;\alpha,1)\leq 1$ is satisfied when $\delta\gg1$. It is not difficult, however, to show that $F(z;\alpha,1)> 1$ for small values of $\delta$. The exact value of $\delta$ where the change occurs can be computed in the following way. First write $F(z;\alpha,1)=V(z)V^{*}(z)$ where
	\[
		V(z)=\dfrac{K(\alpha,z)+K(-\alpha,z)}{K(\alpha,\alpha)+K(-\alpha,\alpha)}.
	\]
The expression $|V(t)|\leq 1$ for each $t\in\R$ is equivalent to
	\[
		4 \sinh\pi\delta\Imaginary(\alpha)  \leq  \sinh2\pi \delta\Imaginary(\alpha) +2\pi \delta\Imaginary(\alpha)
	\]
for $\delta,\Imaginary(\alpha)>0$. Equality is obtained when $\pi\delta\Imaginary(\alpha)\approx 1.0295$.


\subsection*{Vanishing at many points}
	In the introduction we mentioned that our method allows for the construction for majorants what vanish at many distinct points. In this section we elaborate on this and explore some of the analytic properties of  (\ref{manyPts}).\newline
	\indent  Given $N$ points $\alpha_{1},...,\alpha_{N}\in\mU$ represented by ${\boldsymbol \alpha}$, we define the following analogue of $\rho(\alpha,I,\delta)$:
	\[
	 	\rho({\boldsymbol \alpha},I,\delta)=\inf \dint_{-\infty}^{\infty} \lb G(t) - \chi_{I}(t)  \rb dt
	\] 
where the infimum is taken over entire functions $G(z)$ that satisfy: 
	\begin{enumerate}[(i)]
		\item $G(t)\geq \chi_{I}(t)$ for all $t\in\R$,
		\item $\hat{G}(\xi)=0$ whenever $|\xi|>\delta$, and
		\item $G(\alpha_{n})=0$ for n=1,2,...,N.
	\end{enumerate}
By using the function defined in (\ref{manyPts}), we can obtain the following bounds on $\rho({\boldsymbol \alpha},I,\delta)$.
\begin{theorem} \label{rhoTheorem2}
	Let $\delta>0$, $N>0$, $\alpha_{1},...,\alpha_{N}\in\mU$ be represented by ${\boldsymbol \alpha}$, and $I\subset \R$ be an interval. Then
		\begin{equation}
			\delta^{-2} \ll \rho({\boldsymbol \alpha},I,\delta) \ll \delta^{-2N-1}
		\end{equation}
	as $\delta\ra 0$, and
		\begin{equation}
			\rho({\boldsymbol \alpha},I,\delta) \approx \delta^{-1}
		\end{equation}
	as $\delta\ra \infty$. The implied constants are effective and depend only on $N$, ${\boldsymbol \alpha}$ and $I$.
\end{theorem}
 
	\begin{equation}
	      G_{\boldsymbol \alpha}(z)=\dprod_{n=1}^{N}\left( 1+F(z/N;\alpha_{n}/N,-1)  \right).
	\end{equation}
It is easy to see that $G_{\boldsymbol \alpha}(z)$ is an entire function of exponential type  $2\pi\delta$ such that $G_{\boldsymbol \alpha}(x)\geq1$ on the real axis and $G_{\boldsymbol \alpha}(\alpha_{n})=0$ for each $n=1,...,N$. 
Now observe that the following modification of Selberg's majorant 
	\begin{equation}
	      C_{\boldsymbol \alpha}(z)=C(z)G_{\boldsymbol \alpha}(z)
	\end{equation}
has exponential type $4\pi\delta$, $C_{\boldsymbol \alpha}(x)\geq C(x)$ for all real $x$, and $C_{\boldsymbol \alpha}(\alpha_{n})=0$ for each $n=1,...,N$. 
\begin{remark}
 Instead of giving weight $1/N$ to each function depending on $\alpha_{\ell}$, we could instead take a different convex combination using the functions $F(\lambda_{n}z;\lambda_{n}\alpha_{n},-1)$ where $\lambda_{1},...,\lambda_{N}>0$ and $\lambda_{1}+\cdots+\lambda_{N}=1$.
\end{remark}

\begin{lemma}  Let $\varphi\in L^{1}\cap E_{\delta}$ and define $\varphi_{\balpha}(x)=\varphi(x)G^{+}_{\balpha}(x)$. There exists a constant $c=c(\balpha,N)>0$ such that
	\[
		| \hat{\varphi_{\balpha}}(\xi)|\leq|\hat{\varphi}(\xi)|\lp 1+c(\balpha,N) \lp \exp\lb -\pi \ds\min_{\ell}\lb\Imaginary(\alpha_{\ell})\rb\delta/N^{2}  \rb\rp\rp
	\]
as $\delta\ra \infty$ and 
	\[
		|\hat{\varphi_{\alpha}}(0)|\ll \|\varphi\|_{1} \delta^{-2N}
	\]
as $\delta\ra0$.
The implied constants depend on $a,b,$ and $\boldsymbol \alpha$.
\end{lemma}
\begin{proof}
By Lemma \ref{lemmaU} {\it (v)}, we have
	\[
		|1-G_{\balpha}(x)| \ll_{N,\balpha} \exp\lb -\pi \ds\min_{\ell}\lb\Imaginary(\alpha_{\ell})\rb\delta/N^{2}  \rb  
	\]
and using the fact that $|\hat\varphi(\xi)|\leq |\hat{\varphi}(0)|$ we have
	\[
		 \hat{\varphi_{\balpha}}(\xi)=\hat{\varphi}(\xi)\lp 1+O_{N,\balpha} \lp \exp\lb -\pi \ds\min_{\ell}\lb\Imaginary(\alpha_{\ell})\rb\delta/N^{2}  \rb\rp\rp. 
	\]
Now when $\delta\ra 0$ we have

      \begin{eqnarray}
	      & & \dint_{-\infty}^{\infty}\varphi_{\boldsymbol \alpha}(x)dx \\ \nonumber
	      &=& \dint_{-\infty}^{\infty}\varphi(x)dx\\ \nonumber
	      & & +\dsum_{ i_{1}<\cdots<i_{k}}\dint_{-\infty}^{\infty}\varphi(x)F(x/N;\alpha_{i_{1}}/N,-1)\cdots F(x/N;\alpha_{i_{k}}/N,-1)dx \\ \nonumber
	      &\leq& \hat{\varphi}(0)+\|\varphi\|_{1}\dsum_{ i_{1}<\cdots<i_{k}}\dprod_{\ell=1}^{k}\|F(x/N;\alpha_{i_{\ell}}/N,-1)\|_{\infty}. \label{expanded}
      \end{eqnarray}
Lemma \ref{lemmaU}  gives
	\begin{eqnarray*}
		\|F(x/N;\alpha_{i_{\ell}}/N,-1)\|_{\infty} 
		&\leq & \|F(x;\alpha_{i_{\ell}},-1)\|_{\infty} \\
		&\leq& \dfrac{4\pi\delta\Imaginary(\alpha_{i_{\ell}})}{K(\alpha_{\ell},\alpha_{\ell})-\delta}, 
	\end{eqnarray*}
	and by writing 
	\[
		K(\alpha_{\ell},\alpha_{\ell})=\delta+\dsum_{n=2}^{\infty}\dfrac{(2\pi \Imaginary(\alpha_{\ell}))^{2n-2}}{(2n-2)!}\delta^{2n-1}
	\]
	we find that
	\[
		\|F(x/N;\alpha_{i_{\ell}}/N,-1)\|_{\infty} \leq \dfrac{2}{2\pi\Imaginary(\alpha_{i_{\ell}})\delta^{2}},
	\]
	which combined with (\ref{expanded}) gives
	\begin{equation}\label{integrandExpanded}
		\dint_{-\infty}^{\infty}\varphi_{\boldsymbol \alpha}(x)dx 
		\leq
		\hat{\varphi}(0)+N\|\varphi\|_{1}\dsum_{ i_{1}<\cdots<i_{k}}\dprod_{\ell=1}^{k}\dfrac{2}{2\pi\Imaginary(\alpha_{i_{\ell}})\delta^{2}}.
	\end{equation}

It follows that
	\[
		\dint_{-\infty}^{\infty}\varphi_{\boldsymbol \alpha}(x)dx\leq \hat{\varphi}(0)+\|\varphi\|_{1}O(\delta^{-2N})\;\;\;\text{ as }\delta\ra0
	\]
where the implied constant depends on $N$ and $\boldsymbol \alpha$. 
\end{proof}

\subsection{de Branges Spaces}In this section we show how Theorem \ref{interpolateTheorem} can be generalized so that the minimization occurs in a fairly general {\it de Branges space}. A Hilbert space $H$, which is non-trivial and whose elements are {\it entire functions}, is called a {\it de Branges space} if
\begin{enumerate}[(i)]
\item $F\in H$ and $\omega$ is a non-real zero of $F$, then $(z-\ol{\omega})F(z)/(z-\omega)\in H$ and has the same norm as $F$, 
\item $F\in H$ implies $F^{*}\in H$ and has the same norm as $F$, and 
\item for every $\omega\in\C$, then functional $F\mapsto F(\omega)$ is continuous. 
\end{enumerate}
It is a fundamental theorem of de Branges \cite{deBranges} that to each space $H$ there exists an entire function $E(z)$ satisfying the elementary inequality
	\begin{equation}\label{deBrangesCondition}
		|E(\ol{z})|<|E(z)| \;\;\;\;\;\text{ for each }z\in\mU
	\end{equation}
such that the Hilbert space whose elements come from $H$, but whose inner product is given by
	\[
		\al F,G\ar_{E}=\dint_{-\infty}^{\infty} F(t)\ol{G(t)}\dfrac{dt}{|E(t)|^2}
	\]
with induced norm $\|\cdot\|_{E}$, is isometric to $H$. Following \cite{Lagarias}, we will call this function a {\it de Branges function} and we will say that it is {\it strict} if it has no zeros on the real axis. Condition (iii) implies that a de Branges space is a reproducing kernel Hilbert space. We will let $K_{E}(\omega,z)$ denote the corresponding reproducing kernel, which is given by the formula \cite[Theorem 19]{deBranges}
	\begin{equation}\label{deBrangesRepro}
		K_{E}(\omega,z)=\dfrac{B(z)\ol{A(\omega)}-A(z)\ol{B(\omega)}}{\pi(z-\ol{\omega})},
	\end{equation}
where $A(z)=(1/2)(E(z)+E^{*}(z))$ and $B(z)=(i/2)(E(z)-E^{*}(z))$. \newline
\indent Conversely, given an entire function satisfying (\ref{deBrangesCondition}), there exists a de Branges space $H_{E}$ whose elements $F(z)$ satisfy
\begin{enumerate}[(i)]
\item $\| F\|_{E}<\infty$, and 
\item  $F(z)/E(z)$ and $F^{*}(z)/E(z)$ are of {\it bounded type} and non-positive {\it mean type} in $\mU$.
\end{enumerate}
 A function $g(z)$ which is analytic in $\mU$ is said to be of {\it bounded type} in $\mU$ if it can be expressed as the quotient of bounded analytic functions in $\mU$. The {\it mean type} of a function $g(z)$ of bounded type in $\mU$ is the number 
	\[
		\nu(g)=\displaystyle\limsup_{y\ra\infty}y^{-1}\log |g(iy)|
	\]
if $g(z)$ is not identically zero and $-\infty$ if $g\equiv 0$. We can now formulate a generalization of Theorem \ref{interpolateTheorem} for de Branges spaces.

 \begin{theorem} \label{deBrangesInterpolate}
            Let $\alpha\in\mU$, $\beta\in\C$, and $E(z)$ be a de Branges function that is of bounded type in $\mU$. Assume in addition that $K_{E}(\alpha,z)$ and $K_{E}(\ol{\alpha},z)$ are linearly independent.   If $F(z)$ is an entire function of exponential type at most $2\nu(E)$ satisfying
              \begin{enumerate}
                    \item $F(x)\geq0$ for real $x$, and
              \item $F(\alpha)=\beta$,
              \end{enumerate}
            then
              \begin{equation} \label{extrem}
                \dfrac{|\beta|K_{E}(\alpha,\alpha)-\Real(\beta K_{E}(\alpha,\overline{\alpha}))}{K_{E}(\alpha,\alpha)^{2}-|K_{E}(\alpha,\overline{\alpha})|^{2}}\leq \dfrac{1}{2}\dint_{-\infty}^{\infty}F(x)|E(x)|^{-2}dx.
              \end{equation}
            Equality occurs in (\ref{extrem}) if and only if $F(z)=U(z)U^{*}(z)$, where
              \begin{equation}\label{U}
              U(z)= \lambda_{1} K_{E}(\alpha,z) + \lambda_{2}K_{E}(\overline{\alpha},z).
              \end{equation}
            The coefficients $\lambda_{1}$ and $\lambda_{2}$ are given by
              \begin{equation*}
              \lambda_{1}=\dfrac{\gamma\beta K_{E}(\alpha,\alpha)-|K_{E}(\alpha,\overline{\alpha})|}{K_{E}(\alpha,\alpha)^{2}-|K_{E}(\alpha,\overline{\alpha})|^{2}} \;
              \text{ and } \;
                \lambda_{2}=\dfrac{K_{E}(\alpha,\alpha)-\beta K_{E}(\alpha,\overline{\alpha})}{K_{E}(\alpha,\alpha)^{2}-|K_{E}(\alpha,\overline{\alpha})|^{2}}
              \end{equation*}
        where $\gamma=K_{E}(\alpha,\overline{\alpha})|K_{E}(\alpha,\overline{\alpha})|^{-1}$.
      \end{theorem}
      \begin{proof}
      	The proof of Theorem 15 of \cite{HV} shows the existence of a $U(z)\in H_{E}$ such that $F(z)=U(z)U^{*}(z)$. The result follows from Theorem \ref{extremalHilbert} by taking $\bu=K_{E}(\alpha,z)$ and $\bv=K_{E}(\ol{\alpha},z)$. 
      \end{proof}
The condition that $K_{E}(\alpha,z)$ and $K_{E}(\ol{\alpha},z)$ are linearly independent is necessary because of examples such as $E(z)=z+i$. Using (\ref{deBrangesRepro}), we let $A(z)=(1/2)(E(z)+E^{*}(z))=z$, $B(z)=(i/2)(E(z)-E^{*}(z))=-1$, and
	    \[
	    	K_{E}(\alpha,z)=\dfrac{B(z)\ol{A(\alpha)}-A(z)\ol{B(\alpha)}}{\pi(z-\ol{\alpha})}=\dfrac{z-\ol{\alpha}}{\pi(z-\ol{\alpha})}=\pi^{-1}.
	   \]
It follows that $K_{E}(\alpha,z)=K_{E}(\beta,z)$ for every $\alpha,\beta\in\C$. 
In fact $H_{E}\cong\C$ as a Hilbert space.\newline
\indent The following theorem shows that the only obstruction to $K_{E}(\alpha,z)$ begin linearly independent to $K_{E}(\ol{\alpha},z)$ is realized in $E(z)=z+i$. 

	    \begin{theorem} \label{linLemma}
		  Let $E(z)$ be a de Branges function that is of bounded type in $\mU$ and $\alpha\in\mU$. The functions $K_{E}(\alpha,z)$ and $K_{E}(\ol{\alpha},z)$ are linearly independent in $H_{E}$ if, and only if, either of the following conditions hold:
		  \begin{enumerate}
		   \item $E(z)$ has positive mean type, $\nu(E)>0$;
		   \item $E(z)$ has more than one non-real zero.
		  \end{enumerate}
	    \end{theorem}
\noindent  To prove Theorem \ref{linLemma}, we will need the following lemmas.
\begin{lemma}\label{removeRealZerosLemma}
Suppose $E(z)$ is a de Branges function, $S(z)$ is a real entire function with only real zeros, and $\alpha\in\mU$. The function $\tilde{E}(z)=E(z)S(z)$ is a de Branges structure function. There exists a constant $c\neq 0$ such that 
	\[
		K_{\tilde{E}}(\alpha,z)=c K_{\tilde{E}}(\ol{\alpha},z)
	\]
if, and only if, 
	\[
		K_{E}(\alpha,z)=c\dfrac{S(\alpha)}{S(\ol{\alpha})}K_{E}(\ol{\alpha},z).
	\]	
\end{lemma}
\begin{proof}
First we compute $K_{\tilde{E}}(\omega,z)$ in terms of $E(z)$. Notice that 
	\[
			\tilde{A}(z)=\dfrac{\tilde{E}(z)+\tilde{E}^{*}(z)}{2}=S(z)A(z),\text{ and }\;\; \tilde{B}(z)=i\dfrac{\tilde{E}(z)-\tilde{E}^{*}(z)}{2}=S(z)B(z).
	\]	
 By (\ref{deBrangesRepro}) it follows that
    	\[ K_{\tilde{E}}(\omega,z)=S(z)S(\ol{\omega})K_{\tilde{E}}(\omega,z).\]
From this identity we see that there exists a non-zero constant $c$ such that $K_{\tilde{E}}(\alpha,z)=c K_{\tilde{E}}(\ol{\alpha},z)$ if, and only if,
	\[
		K_{E}(\alpha,z)=c\dfrac{S(\alpha)}{S(\ol{\alpha})}K_{E}(\ol{\alpha},z).
	\]	
\end{proof}
\begin{lemma}[P\'olya]\label{polyaLemma}
	Suppose $E(z)$ satisfies the conditions of Theorem \ref{linLemma}. Then there are numbers $b,c\in \C$ satisfying $\Real(b)\geq0$ and $a\neq0$, such that
    	\begin{equation}\label{polyaFactorization}
    		E(z)=az^{k}e^{2\pi i bz }\dprod_{n=1}^{\infty}\lp1 -\dfrac{z}{z_{n}}\rp e^{h_{n}z} 
    	\end{equation}
where $z_{1},z_{2},...$ are the non-zero zeros of $E(z)$ (listed with appropriate multiplicity), $k\geq0$ is the order of the zero at $0$, and  $h_{n}=\Real(z_{n})|z_{n}|^{-2}$. 
\end{lemma}
\begin{proof}
The assumption that $E(z)$ is of bounded type in $\mU$ implies \cite[Problem 34]{deBranges} that $E(z)$ is an entire function of P\'olya class. The theorem then follows from  \cite[Theorem 7]{deBranges} .
\end{proof}

\begin{lemma}\label{strictFactorLemma}
	Let $E(z)$ be a strict de Branges function and let $\alpha\in\mU$. If there is a constant $c\neq0$ such that 
	\[
		K_{E}(\alpha,z)=c K_{E}(\ol{\alpha},z),
	\]
	then there are numbers $b\geq0$, and $a\neq0$, such that
	\[
		E(z)=\begin{cases}
				ae^{2\pi i bz}(z-\omega) & \text{ if } E(\omega)=0\text{ and }\Imaginary(\omega)<0, \\
				ae^{2\pi i bz} & \text{ if } E(z) \text{ has no zeros}.
			\end{cases}
	\]
\end{lemma}
\begin{proof}
If $K_{E}(\alpha,z)=cK_{E}(\ol{\alpha},z)$ for some constant $c\neq 0$, then (\ref{deBrangesRepro}) gives
	\begin{equation}\label{reproScale}
		\dfrac{B(z)\ol{A(\alpha)}-A(z)\ol{B(\alpha)}}{\pi(z-\ol{\alpha})}=c\dfrac{B(z){A(\alpha)}-A(z){B(\alpha)}}{\pi(z-{\alpha})}.
	\end{equation}
From this expression, it follows that  $A(\alpha)=0$ if, and only if, $B(\alpha)=0$. If $A(\alpha)=B(\alpha)=0$, then $E(\alpha)=0$ which is impossible in view of (\ref{deBrangesCondition}). Thus we have both $B(\alpha)\neq0$ and $A(\alpha)\neq 0$. This implies that $A(z),zA(z),B(z),$ and $zB(z)$ are linearly {\it dependent}, i.e. that there are constants $c_{1},..,c_{4}$ (not all zero) such that
	\begin{equation}\label{dependent}
		c_{1}A(z)+c_{2}zA(z)+c_{3}B(z)+c_{4}zB(z)=0.
	\end{equation}
To show that the constants are not all zero we will show that $c_{1}$ and $c_{2}$ cannot both be zero. Observe
	\[
		c_{1}=B(\ol{\alpha})\alpha-cB(\alpha)\ol{\alpha}, \text{ and }\;\;\; c_{2}=B(\ol{\alpha})-cB(\alpha).
	\] 	
If $c_{1}=c_{2}=0$, then $c=B(\ol{\alpha})\alpha/(B(\alpha)\ol{\alpha})=B(\ol{\alpha})/B(\alpha)$. But this is impossible because $\alpha\in\mU$. \newline
\indent Rearranging (\ref{dependent}) yields
	\[
		A(z)=\lp \dfrac{c_{4}z-c_{3}}{c_{2}z-c_{1}}\rp B(z).
	\]
Moreover $E(z)=A(z)-iB(z)$,  which leads us to 
	\[
		E(z)=A(z)\lp 1-i\lp \dfrac{c_{4}z-c_{3}}{c_{2}z-c_{1}}\rp \rp.
	\]
This factorization for $E(z)$, along with (\ref{deBrangesCondition}), implies that $E(z)$ has at most one non-real zero. Therefore, by Lemma \ref{polyaLemma}
	\[
		E(z)=\begin{cases}
				ae^{2\pi i bz}(z-\omega) & \text{ if } E(\omega)=0\text{ and }\Imaginary(\omega)<0, \\
				ae^{2\pi i bz} & \text{ if } E(z) \text{ has no zeros}
			\end{cases}
	\]
    for some $a\neq 0$, and $b\in\C$ with $\Real(b)\geq 0$. By Lemma \ref{removeRealZerosLemma} we may assume $b$ is real.
\end{proof}

\begin{proof}[Proof of Theorem \ref{linLemma}]
  If $E(z)$ is not a strict de Branges function, then $E(z)$ can be written as $E(z)=E_{0}(z)S(z)$ where $S(z)=S^{*}(z)$ and $E_{0}(z)$ has no real zeros. From (\ref{deBrangesCondition}), it follows that $S(z)$ has only real zeros and that $E_{0}(z)$ is a strict de Branges  function. By Lemma \ref{removeRealZerosLemma}, it suffices to prove the lemma when $E(z)$ is a strict de Branges function, and we assume that $E(z)$ is strict throughout the remainder of the proof.\newline
\indent By Lemma \ref{strictFactorLemma}, if $K_{E}(\alpha,z)$ and $K_{E}(\ol{\alpha},z)$ are linearly dependent, then
there exists a $b\geq0$, and $a\neq0$, such that
	\[
		E(z)=\begin{cases}
				ae^{2\pi i bz}(z-\omega) & \text{ if } E(\omega)=0\text{ and }\Imaginary(\omega)<0, \\
				ae^{2\pi i bz} & \text{ if } E(z) \text{ has no non-real zeros}.
			\end{cases}
	\]
In light of this structure, it suffices to show that $K_{E}(\alpha,z)$ and $K_{E}(\ol{\alpha},z)$ are linearly dependent if, and only if, $b=0$.
    
 If $b=0$, then $H_{E}$ is one dimensional and $K_{E}(\alpha,z)$ and $K_{E}(\alpha,z)$ must be linearly dependent. \newline
  \indent  Suppose $b>0$ and $\omega$ is the single non-real zero of $E(z)$. The function $G(z)=1$ is then in $H_{E}$ since $\nu(1/E)=0-\nu(E)\leq0$ and $|t-\omega|^{-2}$ is integrable. Similarly, the function $H(z)=e^{-2\pi i b z}$ is in $H_{E}$. Now suppose, by way of contradiction, that there is a non-zero constant $c\in\C$ such that $F(\alpha)=cF(\ol{\alpha})$ for all $F\in H_{E}$. The constant $c$ must be equal to $1$, because \[ 1=G(\alpha)=cG(\ol{\alpha})=c.\] The function $H(z)$ is in $H_{E}$ but $H(\alpha)\neq cH(\ol{\alpha})=H(\ol{\alpha})$, a contradiction. \newline    
\indent  If $E(z)$ has no zeros, then we must have $b>0$, by (\ref{deBrangesCondition}). In this case, the reproducing kernel for $H_{E}$ is given by 
	    \begin{equation}
		  K(\omega,z)=\dfrac{\sin2\pi b(z-\overline{\omega})}{\pi(z-\overline{\omega})}.
	    \end{equation}
    $K(\alpha,z)$ and $K(\ol{\alpha},z)$ are linearly independent in $H_{E}$ if, and only if,  
    \[ K(\alpha,\alpha)^{2}-|K(\alpha,\ol{\alpha})|^{2}=\lb\dfrac{\sinh 4\pi b\Imaginary(\alpha)}{2\pi\Imaginary(\alpha)}\rb^{2}-4b^{2}\neq0.\]
    This holds if, and only if, $\Imaginary(\alpha)>0$.
    \end{proof}

\subsection{Trigonometric Polynomials}
{\it In this final section we prove an analogue of Theorem \ref{interpolateTheorem} for trigonometric polynomials.}\newline
\indent Let $N\geq 1$ and $\mathscr{P}_{N}\subset \C[z]$ be the complex vector space of polynomials of degree at most $N$. Define an inner product $\al\cdot,\cdot\ar$ on $\mathscr{P}_{N}$ by 
      \[
	    \al p,q \ar=\dint_{S^{1}}p(\theta)\ol{q(\theta)}d\sigma(\theta)
      \]
where $S^{1}=\lb z\in\C: |z|=1\rb$ and $\sigma$ is the Haar probability measure on $S^{1}$. Let $\|\cdot\|$ be the norm induced by $\al\cdot,\cdot\ar$ and $e(t)=e^{2\pi i t}$. To simplify notation, we will define a function $L:\C\times\Z^{+}\ra \R$ by
	  \[
		L(\omega,M)=\dsum_{m=0}^{M}|\omega|^{m}=\begin{cases}
											\dfrac{1-|\omega|^{M}}{1-|\omega|} & \text{ if } |\omega|\neq1 \\
											M+1 & \text{ if } |\omega|=1.
											\end{cases}
	  \]

	  \begin{lemma}\label{trigLemma}
		The space $(\mathscr{P}_{N},\al\cdot,\cdot\ar)$ is a reproducing kernel Hilbert space with reproducing kernel $K:\C\times\C\ra\C$ given by
			  \[
				  K(\omega,z)=\dsum_{n=0}^{N}z^{n}\ol{\omega}^{n}.
			  \]
		Furthermore, if $\alpha\neq0$, the functions $z\mapsto K(\alpha,z)$ and $z\mapsto K(1/\ol{\alpha},z)$ are linearly independent if, and only if, $\alpha\not\in S^{1}$.
	  \end{lemma}
	  \begin{proof}
		The form of the reproducing kernel is a standard fact, but we note that it doesn't depend on the choice of orthonormal basis. We need only verify the latter statement.  The functions in question are linearly independent if and only if 
			    \begin{equation}\label{condition1}
				  \|K(\cdot,\alpha)\|^{2}\|K(\cdot,1/\ol{\alpha})\|^{2}-|\al K(\cdot,\alpha),K(\cdot,1/\ol{\alpha})\ar|^{2}\not= 0.
			    \end{equation}
		 Explicitly, these terms are given by \[ K(\alpha,\alpha)=L(|\alpha|,N),~\;\;K(\alpha,1/\ol{\alpha})=N+1\text{, and }~ \;\;K(1/\ol{\alpha},1/\ol{\alpha})=L(|\alpha|^{-1},N).\] 
		 Now (\ref{condition1}) can be rewritten as
			    \[
				  L(|\alpha|,N)L(|\alpha|^{-1},N)-(N+1)^{2}\neq 0.
			    \]
		But by the arithmetic-geometric mean inequality we have
			    \[
				    (N+1)^{-2}\lb\dsum_{n=0}^{N}|\alpha|^{n}\rb\lb\dsum_{n=0}^{N}|\alpha|^{-n}\rb \geq \lb\dprod_{n=0}^{N}|\alpha|^{n}\rb^{1/(N+1)}\lb\dprod_{m=0}^{N}|\alpha|^{-m}\rb^{1/(N+1)}=1.
			    \]
		Equality holds if, and only if, $|\alpha|^{m}=|\alpha|^{n}$ for $n,m=1,2,...,N$ which implies $|\alpha|=1$.
	  \end{proof}

      \begin{corollary} \label{trig}
            Suppose $\alpha\neq0$, $\alpha\not\in S^{1}$, $\beta\in\C$, and $N\geq1$.
            If $F(z)$ is a Laurent polynomial of degree at most $N$,
              \begin{enumerate}
                    \item $F(\theta)\geq0$ for $\theta\in S^{1}$, and
              \item $F(\alpha)=\beta$,
              \end{enumerate}
            then
              \begin{equation} \label{extrem}
                \dfrac{|\beta|L(\alpha,N)^{1/2}L(\alpha^{-1},N)^{1/2}-(N+1)\Real(\beta )}{L(\alpha,N)L(\alpha^{-1},N)-(N+1)^{2}}\leq \dfrac{1}{2}\dint_{S^{1}}F(\theta)d\sigma(\theta).
              \end{equation}
            Equality occurs in (\ref{extrem}) if, and only if, $F(z)=p(z)p^{*}(z)$, where
              \begin{equation}\label{U}
              p(z)= \lambda_{1} K(\alpha,z) + \lambda_{2}K(1/\overline{\alpha},z).
              \end{equation}
            The coefficients $\lambda_{1}$ and $\lambda_{2}$ can be explicitly computed in terms of $K,$ $\alpha$, $\beta$ and $N$.
      \end{corollary}
      \begin{proof}
	   By Fej\'er's theorem, there exists a $p\in\mathscr{P}_{N}$ such that \[F(z)=p(z)p^{*}(z),\] where $p^{*}(z)=\ol{p(1/\ol{z})}$. The result follows from Lemma \ref{trigLemma} and Lemma \ref{extremalHilbert}.
      \end{proof}

\nocite{*}
\bibliographystyle{plain}
\bibliography{selberg_vanishing}	

\begin{thebibliography}{10}

\bibitem{BMV}
Jeffrey~T. Barton, Hugh~L. Montgomery, and Jeffrey~D. Vaaler.
\newblock Note on a {D}iophantine inequality in several variables.
\newblock {\em Proc. Amer. Math. Soc.}, 129(2):337--345 (electronic), 2001.

\bibitem{B}
Ralph~Philip Boas, Jr.
\newblock {\em Entire functions}.
\newblock Academic Press Inc., New York, 1954.

\bibitem{MR0286773}
E.~Bombieri.
\newblock A note on the large sieve.
\newblock {\em Acta Arith.}, 18:401--404, 1971.

\bibitem{Bom}
E.~Bombieri.
\newblock personal communication, 2012.

\bibitem{MR2739041}
Emanuel Carneiro and Vorrapan Chandee.
\newblock Bounding {$\zeta(s)$} in the critical strip.
\newblock {\em J. Number Theory}, 131(3):363--384, 2011.

\bibitem{CL}
Emanuel Carneiro and Friedrich Littmann.
\newblock Bandlimited approximations to the truncated {G}aussian and
  applications.
\newblock {\em Constr. Approx.}, 38(1):19--57, 2013.

\bibitem{CL2}
Emanuel Carneiro and Friedrich Littmann.
\newblock Entire {A}pproximations for a {C}lass of {T}runcated and {O}dd
  {F}unctions.
\newblock {\em J. Fourier Anal. Appl.}, 19(5):967--996, 2013.

\bibitem{CLV}
Emanuel Carneiro, Friedrich Littmann, and Jeffrey~D. Vaaler.
\newblock Gaussian subordination for the beurling-selberg extremal problem.
\newblock {\em Trans. Amer. Math. Soc. (to appear)}, 2012.

\bibitem{CV2010}
Emanuel Carneiro and Jeffrey~D. Vaaler.
\newblock Some extremal functions in {F}ourier analysis. {II}.
\newblock {\em Trans. Amer. Math. Soc.}, 362(11):5803--5843, 2010.

\bibitem{CVIII}
Emanuel Carneiro and Jeffrey~D. Vaaler.
\newblock Some extremal functions in {F}ourier analysis. {III}.
\newblock {\em Constr. Approx.}, 31(2):259--288, 2010.

\bibitem{MR2781205}
Vorrapan Chandee and K.~Soundararajan.
\newblock Bounding {$\vert \zeta(\frac12+it)\vert $} on the {R}iemann
  hypothesis.
\newblock {\em Bull. Lond. Math. Soc.}, 43(2):243--250, 2011.

\bibitem{deBranges}
Louis de~Branges.
\newblock {\em Hilbert spaces of entire functions}.
\newblock Prentice-Hall Inc., Englewood Cliffs, N.J., 1968.

\bibitem{GV}
S.~W. Graham and Jeffrey~D. Vaaler.
\newblock A class of extremal functions for the {F}ourier transform.
\newblock {\em Trans. Amer. Math. Soc.}, 265(1):283--302, 1981.

\bibitem{HV}
Jeffrey~J. Holt and Jeffrey~D. Vaaler.
\newblock The {B}eurling-{S}elberg extremal functions for a ball in {E}uclidean
  space.
\newblock {\em Duke Math. J.}, 83(1):202--248, 1996.

\bibitem{hormander}
Lars H{\"o}rmander.
\newblock {\em The analysis of linear partial differential operators. {I}},
  volume 256 of {\em Grundlehren der Mathematischen Wissenschaften [Fundamental
  Principles of Mathematical Sciences]}.
\newblock Springer-Verlag, Berlin, 1983.
\newblock Distribution theory and Fourier analysis.

\bibitem{Lagarias}
Jeffrey~C. Lagarias.
\newblock Hilbert spaces of entire functions and {D}irichlet {$L$}-functions.
\newblock In {\em Frontiers in number theory, physics, and geometry. {I}},
  pages 365--377. Springer, Berlin, 2006.

\bibitem{MR1452480}
Xian-Jin Li.
\newblock On reproducing kernel {H}ilbert spaces of polynomials.
\newblock {\em Math. Nachr.}, 185:115--148, 1997.

\bibitem{MR2117219}
Xian-Jin Li.
\newblock A note on the weighted {H}ilbert's inequality.
\newblock {\em Proc. Amer. Math. Soc.}, 133(4):1165--1173 (electronic), 2005.

\bibitem{LV}
Xian-Jin Li and Jeffrey~D. Vaaler.
\newblock Some trigonometric extremal functions and the {E}rd{\H o}s-{T}ur\'an
  type inequalities.
\newblock {\em Indiana Univ. Math. J.}, 48(1):183--236, 1999.

\bibitem{LS}
F.~{Littmann} and M.~{Spanier}.
\newblock {Extremal functions with vanishing condition}.
\newblock {\em ArXiv e-prints}, November 2013.

\bibitem{Litt05}
Friedrich Littmann.
\newblock Entire approximations to the truncated powers.
\newblock {\em Constr. Approx.}, 22(2):273--295, 2005.

\bibitem{Litt06}
Friedrich Littmann.
\newblock Entire majorants via {E}uler-{M}aclaurin summation.
\newblock {\em Trans. Amer. Math. Soc.}, 358(7):2821--2836 (electronic), 2006.

\bibitem{Littmann2006}
Friedrich Littmann.
\newblock One-sided approximation by entire functions.
\newblock {\em J. Approx. Theory}, 141(1):1--7, 2006.

\bibitem{Litt08}
Friedrich Littmann.
\newblock Interpolation and approximation by entire functions.
\newblock In {\em Approximation theory {XII}: {S}an {A}ntonio 2007}, Mod.
  Methods Math., pages 243--255. Nashboro Press, Brentwood, TN, 2008.

\bibitem{Litt09}
Friedrich Littmann.
\newblock Zeros of {B}ernoulli-type functions and best approximations.
\newblock {\em J. Approx. Theory}, 161(1):213--225, 2009.

\bibitem{Litt11}
Friedrich Littmann.
\newblock {$L^1$}-approximation to {L}aplace transforms of signed measures.
\newblock {\em J. Approx. Theory}, 163(10):1492--1508, 2011.

\bibitem{Littmann2013}
Friedrich Littmann.
\newblock Quadrature and extremal bandlimited functions.
\newblock {\em SIAM J. Math. Anal.}, 45(2):732--747, 2013.

\bibitem{Logan77}
B.F. Logan.
\newblock {B}andlimited functions bounded below over an interval.
\newblock {\em Notices Amer. Math. Soc.}, 24:A--331, 1977.

\bibitem{M1978}
Hugh~L. Montgomery.
\newblock The analytic principle of the large sieve.
\newblock {\em Bull. Amer. Math. Soc.}, 84(4):547--567, 1978.

\bibitem{pp}
M.~Plancherel and G.~P{\'o}lya.
\newblock Fonctions enti\`eres et int\'egrales de fourier multiples.
\newblock {\em Comment. Math. Helv.}, 10(1):110--163, 1937.

\bibitem{MR1307384}
Marvin Rosenblum and James Rovnyak.
\newblock {\em Topics in {H}ardy classes and univalent functions}.
\newblock Birkh\"auser Advanced Texts: Basler Lehrb\"ucher. [Birkh\"auser
  Advanced Texts: Basel Textbooks]. Birkh\"auser Verlag, Basel, 1994.

\bibitem{Sel}
Atle Selberg.
\newblock {\em Collected papers. {V}ol. {II}}.
\newblock Springer-Verlag, Berlin, 1991.
\newblock With a foreword by K. Chandrasekharan.

\bibitem{Stein57}
E.~M. Stein.
\newblock Functions of exponential type.
\newblock {\em Ann. of Math. (2)}, 65:582--592, 1957.

\bibitem{SW}
Elias~M. Stein and Guido Weiss.
\newblock {\em Introduction to {F}ourier analysis on {E}uclidean spaces}.
\newblock Princeton University Press, Princeton, N.J., 1971.
\newblock Princeton Mathematical Series, No. 32.

\bibitem{Vaa}
J.~Vaaler.
\newblock personal communication.
\newblock 2011.

\bibitem{V}
Jeffrey~D. Vaaler.
\newblock Some extremal functions in {F}ourier analysis.
\newblock {\em Bull. Amer. Math. Soc. (N.S.)}, 12(2):183--216, 1985.

\end{thebibliography}
\end{document}